\documentclass[a4paper,12pt,reqno]{amsart}
\author{Julia Brandes}
\title[linear spaces on hypersurfaces]{On the number of linear spaces on hypersurfaces with a prescribed discriminant}
\address{Mathematical Sciences, Chalmers Institute of Technology and University of Gothenburg, 412 96 G{\"o}teborg, Sweden}%\\
\email{brjulia@chalmers.se}

\usepackage[T1]{fontenc}
\usepackage[latin1]{inputenc}
\usepackage{latexsym}
\usepackage[arrow, matrix, curve]{xy}
\usepackage[dvips]{graphicx}
\usepackage{epsfig}
\usepackage{bm}
\usepackage{amssymb}
\usepackage{amsmath}
\usepackage{verbatim}
\usepackage{amsthm}
\usepackage{enumerate}
\usepackage{mathrsfs}
\usepackage[hmargin=3cm,vmargin=3cm]{geometry}
\usepackage{mathtools}

\newtheorem{thm}{Theorem}

\newtheorem{lem}{Lemma}

\numberwithin{equation}{section}
\numberwithin{thm}{section}
\numberwithin{lem}{section}

\theoremstyle{definition}

\def\Z{\mathbb Z}
\def\Q{\mathbb Q}
\def\R{\mathbb R}

\def\T{\mathbb T}

\def\B#1{\mathbf{#1}}
\def\ba{\bm{\alpha}}
\def\bb{\bm{\beta}}

\def\F#1{\mathfrak{#1}}
\def\cal#1{\mathcal{#1}}
\def\sc#1{\mathscr{#1}}
\def\D{\;\mathrm{d}}

\def\ol#1{\overline{\B{#1}}}

\def\U#1{\underline{\mathbf{#1}}}
\def\UU#1{\underline{\bm{#1}}}

\def\Ua{\underline{\bm{\alpha}}}
\def\UA{\underline{\mathbf{a}}}

\def\Ub{\underline{\bm{\beta}}}
\def\Ug{\underline{\bm{\gamma}}}
\def\a0{\alpha_0}

\def\eps{\varepsilon}
\def\tq{\tilde{q}}

\DeclareMathOperator{\lcm}{lcm}

\DeclareMathOperator{\card}{Card}

\DeclareMathOperator{\vol}{vol}
\DeclareMathOperator{\sing}{Sing}

\renewcommand\le{\leqslant}
\renewcommand\ge{\geqslant}

\begin{document}

\begin{abstract}
  For a given form $F\in \Z[x_1,\dots,x_s]$ we apply the circle method in order to give an asymptotic estimate of the number of $m$-tuples $\B x_1, \dots, \B x_m$ spanning a linear space on the hypersurface $F(\B x) = 0$ with the property that $\det ( (\B x_1, \dots, \B x_m)^t \, (\B x_1, \dots, \B x_m)) = b$. This allows us in some measure to count rational linear spaces on hypersurfaces whose underlying integer lattice is primitive.
\end{abstract}

\maketitle

\section{Introduction}

Let $F \in \Z[x_1,\dots,x_s]$ be a form of degree $d$. In previous work \cite{FRF}, \cite{FRF2} we investigated the number of rational linear spaces of dimension $m$ contained in the hypersurface given by $F(\B x)=0$. Let $N_m(P)$ denote the number of $\B x_1, \dots, \B x_m \in \Z^s$, $|x_{ij}| \le P \; (1 \le i \le m, 1 \le j \le s)$, satisfying
\begin{align}\label{1.1}
    F(\B x_1 t_1 + \dots + \B x_m t_m)=0 \quad \text{identically in $t_1, \dots, t_m$},
\end{align}
and set $r = \binom{m+d-1}{d}$. Theorem~1.1 of \cite{FRF} and Theorem~1.3 of \cite{FRF2} show that there exists a positive parameter $\delta$ such that
\begin{align}\label{1.3}
    N_m(P) = P^{ms-rd} \chi_{\infty} \prod_{p \text{ prime}} \chi_p + O(P^{ms-rd-\delta})
\end{align}
for some non-negative constants $\chi_{\infty}$ and $\chi_p$ characterising the density of solutions over the local fields $\R$ and $\Q_p$, respectively, provided that
\begin{align*}
    s - \dim \sing F > 2^d r (d-1).
\end{align*}
It is, however, apparent that the strategy of counting linear spaces by analysing Equation \eqref{1.1} is susceptible to double-counting in several ways. Rational linear spaces can be viewed as integer lattices, but in order to count lattices via counting sets of generators, we need tools both to identify primitive lattices and to account for the multiplicity factor by which each individual lattice is counted. The objective of this memoir is to make a step in this direction by counting integer lattices contained in hypersurfaces that have a fixed lattice discriminant. If $L' \subseteq L$ is a sublattice, then the discriminant of $L$ will divide the discriminant of $L'$. It follows that the $m$-dimensional sublattices of $\Z^s$ counted by $N_m(P)$ with prime discriminant will be primitive unless they are oriented along coordinate axes, in which case they are primitive if and only if they are unimodular.

Suppose that $(\B x_1, \dots, \B x_m)$ is a solution to \eqref{1.1}, then we may consider the lattice $X$ spanned by $\B x_1, \dots, \B x_m$. This is a sublattice of $\Z^s$ of dimension at most $m$ and with lattice discriminant $\cal D(X) = \sqrt{\det((\B x_1, \dots, \B x_m) ^t \,(\B x_1, \dots, \B x_m))}$. Write $D(\B x_1, \dots, \B x_m) = \cal D(X)^2$, so that $D$ is a homogeneous polynomial of degree $2m$ in the $ms$ variables $(\B x_1, \dots, \B x_m) \in \Z^{ms}$, and denote by $N_m(P;b)$ the number of $(\B x_1, \dots, \B x_m) \in \Z^{ms}$ counted by $N_m(P)$ that additionally satisfy the equation
\begin{align}\label{1.4}
    D(\B x_1, \dots, \B x_m) &= b.
\end{align}
The main result of this paper is an asymptotic estimate for $N_m(P;b)$.
\begin{thm}\label{mainthm}
    Let $F \in \Z[x_1, \dots, x_s]$ be a smooth form of degree $d \ge 2$, and let $m \geqslant 2$ be a positive integer. Furthermore, suppose that
    \begin{align*}
        s > \max\left\{ 2^{d-1}(6m+3r+2)(d-1) + 2^{2m}m, 2^{d-1}(d-1)(3r+2+2m(r+3)/d)\right\}
    \end{align*}
    if $d>2m$, and
    \begin{align*}
        s> 2^{d-1}rd +2^{2m-1}(2+rd)(2m-1)
    \end{align*}
    if $d<2m$.
    Then there exists a $\delta>0$ such that
    \begin{align}\label{1.5}
        N_m(P;b) = P^{ms-rd-2m}\chi_{\infty}(b,P) \prod_{p \text{ prime}} \chi_p(b) + O(P^{ms-rd-2m-\delta})
    \end{align}
    with non-negative factors $\chi_{\infty}(b,P)$ and $\chi_p(b)$ that characterise the density of solutions of the system over the local fields $\R$ and $\Q_p$, respectively.
\end{thm}
More general versions of each of the cases of Theorem~\ref{mainthm} are available below (see Theorems \ref{T4.1} and \ref{T7.1}) that somewhat relax the requirement that $F$ should be smooth. The glaring omission here is of course the case $d=2m$; while the analytical aspects of the treatment of this case are in fact more conventional than in the situation when $d \neq 2m$ and largely follow the arguments of \cite{birch} and \cite{FRF}, the geometry creates additional difficulties when the singularities of $F$ interfere with the discriminant equation. We plan to resolve this issue in future work.

We will prove Theorem~\ref{mainthm} by the circle method via a combination of the ideas presented in our own work \cite{FRF} and recent work by Browning and Heath-Brown \cite{bhb}. Observe that the latter can be applied directly to the simultaneous equations \eqref{1.1} and \eqref{1.4}, yielding conditions of the shape $s>s_0(d,m)$, where $s_0(d,m) \ll 2^d r^2 d$ for $d >m$ and $s_0(d,m) \ll 2^{2m}r^2m$ when $2m>d$. In particular, these bounds grow quadratically in the parameter $r$, which is itself of size $m^d$. In contrast, Theorem~\ref{mainthm} exhibits linear growth in $r$.

It is an obvious question whether or not, and in what circumstances, the local factors in Theorem~\ref{mainthm} are positive, so that the formula in \eqref{1.5} provides an honest asymptotic estimate. Ideally, one might be tempted to speculate that there might be a way of relating each factor $\chi_p(b)$ in \eqref{1.5} with the corresponding factor $\chi_p$ in \eqref{1.3}, but unfortunately it seems highly unlikely for such a relation to hold in general. Nonetheless, we are able to say something about the factors. An argument by Schmidt \cite[\S 2]{cubIV} shows that $\chi_{\infty}(b,P)>0$, provided that that the variety described by \eqref{1.1} and \eqref{1.4} has a positive $(ms-r-1)$-dimensional volume inside $[-1,1]^{ms}$. Similarly, one can show by arguments due to Schmidt (see \cite[Lemma 2 and \S 11]{schmidtquad} and \cite[\S 2]{cubIV}) in combination with a bound of Wooley on $p$-adic solubility \cite[Corollary 1.1]{tdw-local} that $\chi_p(b)$ is positive whenever
\begin{align*}
    s\ge\max\{&2^{d-1}(d-1)(r+1)((r+1)d^2)^{2^{d-1}}, \\
    &2^{2m-1}(2m-1)(r+1)((r+1)(2m)^2)^{2^{2m-1}}\},
\end{align*}
though we note that for generic forms $F$ this bound may be improved somewhat by earlier work of the author (see \cite[Corollary 1]{Hensel}).

In the context of Theorem~1.1, the choice of $b=0$ is somewhat distinguished, as it corresponds to the number of choices for $\B x_1, \dots, \B x_m$ satisfying \eqref{1.1} that have a linear independence relation between them. Such solutions span linear spaces of dimension at most $m-1$, and thus represent the `degenerate' solutions to \eqref{1.1}. As might be expected, results controlling the number of such degenerate solutions can be obtained by much simpler means. Lewis and Schulze-Pillot \cite[p. 283]{lsp} addressed the issue perfunctorily by showing that the set of linear spaces of dimension at most $m/2$ has itself dimension smaller than $ms-rd$ whenever $s>m/2 +2rd/m$, which is sufficient for applications requiring only an asymptotic dependence on $m$ (see e.g. Lewis and Schulze-Pillot \cite[Theorems 1 and 3]{lsp}, Dietmann \cite[Theorem~2]{rd08}, and Brandes \cite[Theorem~1.3]{FRF}). However, even a precise statement can easily be established by observing that the main term of $N_m(P)$ exceeds that of $N_{m-1}(P)$ as soon as $s>\binom{d+m-2}{d-1} d$, a much weaker condition than what is required in Theorem~ \ref{mainthm}. Nonetheless, one could ask even in this setting how Theorem~\ref{mainthm} compares with other analytic methods in showing that $N_m(P;0) = o(N_m(P))$. The conclusion of Theorem~\ref{mainthm} is stronger than necessary in that it saves an additional amount of $P^{2m}$ over what is needed for a non-trivial result. Marmon \cite{mar} recently showed that non-trivial upper bounds can be established even when the number of variables is smaller than what is required for an asymptotic formula. However, in order to save the required amount, his methods still yield bounds on the number of variables that grow quadratically in $r$, though one can potentially improve on this by optimising his treatment for situations involving linear spaces.

This question can be interpreted in somewhat different terms in the context of counting matrices with a fixed determinant. Whilst the determinant is in many ways the most natural measure of the size of a matrix, its hyperbolic nature renders it unsuitable as a height function. Hence in settings that require a finite-volume height function one typically resorts to height functions that increase with the size of the coefficients, and this raises the question of whether the two can be related. Duke, Rudnick and Sarnak \cite[example 1.6]{DRS} provide a count of matrices of bounded euclidean height with a given non-zero determinant. Our Theorem \ref{mainthm} can be viewed as a generalisation of their result in the sense that we count matrices whose constituting columns lie not in the affine space but on a hypersurface. Obviously, such a question can sensibly be asked only if the dimension of the linear space is smaller than the dimension of the embedding variety, and in practice we require the much more stringent condition that the variety contain a sufficiently high-dimensional linear space on which to perform such a count.\\

As a result of the methods applied, the proof of Theorem~\ref{mainthm} naturally consists of two essentially self-contained parts, as the cases $d>2m$ and $d<2m$ need separate treatment. We will consider the situation when $d<2m$ in \S\S 2--4, and turn to the structurally similar but technically slightly more demanding case $d>2m$ in \S\S 5--7.

Throughout the paper, the following conventions will be observed. Every statement involving the letter $\eps$ is true for all $\eps>0$, and consequently no effort will be made to trace the `value' of $\eps$. Statements involving vectors should be read entry-wise, so $|\B x| \le P$ means $|x_i| \le P$ for all components $x_i$ of $\B x$. Similarly, we write $(\B a, b)$ for the greatest common divisor of all entries $a_i$ and $b$. Expressions like $\sum_{n=1}^x f(n)$, where $x$ may or may not be an integer, are always to be interpreted as $\sum_{1 \leqslant n \leqslant x} f(n)$. We will commonly write $\T = \R/\Z$. Finally, we write $e(x) = e^{2 \pi i x}$ and use the Landau and Vinogradov notation extensively. All implied constants are allowed to depend on $s$, $d$, $m$ and the coefficients of $F$, but are independent of $P$, which is always taken to be a large integer.\\\hspace*{\fill}

\textbf{Acknowledgements.} The author is very grateful to Oscar Marmon for reading an earlier draft version of this article and for making available a preprint version of \cite{mar}, and to Tim Browning for valuable comments.

\section{The case $2m>d$: Weyl's inequality}

Let $\Phi$ be the symmetric $d$-linear form associated to $F$, and write $J$ for the set of $d$-tuples $\B j \in \{1,\dots,m\}^d$ neglecting order, so that $\card J = r$. Then we have
\begin{align*}
    F(\B x_1 t_1 + \dots + \B x_m t_m)=\sum_{\B j \in J}A(\B j) t_{j_1} \cdot \ldots \cdot t_{j_d} \Phi(\B x_{j_1}, \dots, \B x_{j_d})
\end{align*}
with suitable combinatorial factors $A(\B j) \in \Z/d!$. Writing $\ol x = (\B x_1, \dots, \B x_m)$ and $\Phi_{\B j}(\ol x) = A(\B j)\Phi(\B x_{j_1}, \dots, \B x_{j_d})$, we see that \eqref{1.1} holds for precisely those $\ol x$ that satisfy
\begin{align*}
    \Phi_{\B j}(\ol x) = 0 \qquad (\B j \in J).
\end{align*}
Let $ \ba=(\alpha_{\B j})_{\B j \in J}$ and $\Ua = (\ba, \alpha_0) \in \T^r \times \T$, and write
\begin{align*}
  \F F(\ol x; \ba)  = \sum_{\B j \in J} \alpha_{\B j} \Phi_{\B j}(\ol x) \quad \text{and} \quad \F F_0(\ol x; \Ua) = \F F(\ol x; \ba) + \alpha_0 D(\ol x),
\end{align*}
then we may define the exponential sum as
\begin{align*}
    T_P(\Ua)=\sum_{|\ol x| \leqslant P} e(\F F_0(\ol x; \Ua)).
\end{align*}
Recalling the standard orthogonality relations from Fourier analysis, the counting function $N_m(P;b)$ is now given by
\begin{align*}
  N_m(P;b) = \int_{\T^{r+1}} T_P(\Ua) e(-\alpha_0 b) \D \Ua.
\end{align*}
Our first task is to bound $T_P(\Ua)$ from above.
Define the discrete differencing operator via its action on a test polynomial $G \in \Z[\B x_1, \dots, \B x_m]$ as
\begin{align}\label{2.1}
  \Delta_{i, \B h}G(\ol x) =  G(\B x_1, \dots, \B x_i + \B h, \dots, \B x_m) - G(\B x_1, \dots, \B x_i , \dots, \B x_m).
\end{align}
The following is an easy modification of Lemma~3.1 of \cite{FRF}.

\begin{lem}\label{L2.1}
  Let $1 \leqslant k \leqslant 2m$ and $j_l$, $1 \leqslant l \leqslant k$, be integers with $1 \leqslant j_l \leqslant m$. Then
  \begin{align*}
      |T_P(\Ua)|^{2^k} \ll P^{((2^k-1)m-k)s} \sum_{\B h_1, \dots, \B h_k } \sum_{\ol x} e \left( \Delta_{j_1, \B h_1} \cdots \Delta_{j_k, \B h_k} \F F_0(\ol x; \Ua)\right),
  \end{align*}
    where all sums range over suitable boxes of sidelength at most $2P$.
\end{lem}
This auxiliary lemma allows us to establish a tripartite Weyl inequality in terms of $\alpha_0$, but first we need to acquire a better understanding of the action of $\Delta$ on $D$. Observe that every vector $\B x_i$ occurs quadratically within $D(\ol x)$. It follows that the expression $\Delta_{j_1, \B h_1} \cdots \Delta_{j_{2m-1}, \B h_{2m-1}} D(\ol x; \Ua)$ vanishes whenever an entry of $(j_1, \dots,  j_{2m-1})$ occurs at least thrice, and otherwise depends only on $\B x_{j_{2m}}$, where $j_{2m}$ is the unique parameter occurring only once in $(j_1, \dots, j_{2m-1})$. Let $\sigma$ be a permutation of $(1,2,\dots,m)$ and suppose that $\B j = (\sigma(1), \dots, \sigma(m), \sigma(1), \dots, \sigma(m))$. Let further $\cal T \subseteq S_{2m}$ denote the group generated by the transpositions $(k, m+k)$, and write
\begin{align*}
    &{\det}_{\cal T} \left( (\B h_1, \dots, \B h_m)^t \, (\B h_{m+1}, \dots, \B h_{2m})\right) \\
    & \quad= \sum_{\tau \in \cal T} \det \left( (\B h_{\tau(1)}, \dots, \B h_{\tau(m)})^t \,(\B h_{\tau(m+1)}, \dots, \B h_{\tau(2m)})\right),
\end{align*}
then for fixed $\B h_1, \dots, \B h_{2m-1}$ we have
\begin{align*}
    \Delta_{j_1, \B h_1} \cdots  \Delta_{j_{2m-1},\B h_{2m-1}} D(\ol x) & =  {\det}_{\cal T} \left( (\B h_1, \dots, \B h_m)^t \, (\B h_{m+1}, \dots, \B h_{2m-1}, \B x_{\sigma(m)})\right) \nonumber\\
    & \qquad + \text{ terms independent of $\ol x$.}
\end{align*}
Define the $(2m-1)$-linear forms $d_{n}$  for $1 \leqslant n \leqslant s$ via
\begin{align}\label{2.2}
    d_n(\B x^{(1)}, \dots, \B x^{(2m-1)}) = {\det}_{\cal T} \left( (\B x^{(1)}, \dots, \B x^{(m)})^t \, (\B x^{(m+1)}, \dots, \B x^{(2m-1)}, \B e_n)\right) ,
\end{align}
where $\B e_n$ denotes the $n$-th unit vector in $\Z^s$.
In this notation, the Weyl-type lemma is as follows.

\begin{lem}\label{L2.2}
    Suppose that $l>0$ and $0 <\eta \leqslant 1$ are parameters and $\Ua \in \T^{r+1}$, then one of the following holds.
    \begin{enumerate}[(A)]
        \item We have $|T_P(\Ua)| \ll P^{ms-l\eta}$, or
        \item there are integers $0 \leqslant a_0 < q_0 \ll P^{(2m-1)\eta}$ satisfying $|\alpha_0 q_0 - a_0| \ll P^{-2m+(2m-1)\eta}$, or
        \item the number of integer vectors $\B h_1, \dots, \B h_{2m-1} \in [-P^{\eta}, P^{\eta}]^s$ satisfying
            \begin{align}\label{2.3}
                d_n(\B h_1, \dots, \B h_{2m-1}) = 0 \qquad (1 \leqslant n \leqslant s)
            \end{align}
            is asymptotically larger than $(P^{\eta})^{(2m-1)s - 2^{2m-1}l - \eps}$.
    \end{enumerate}
\end{lem}

\begin{proof}
    This is only a slight modification of Lemma~2.5 of Birch \cite{birch}. Applying Lemma~\ref{L2.1} with $k=2m-1$ yields
    \begin{align*}
        |T_P(\Ua)|^{2^{2m-1}} \ll P^{(2^{2m-1}m-3m+1)s} \sum_{\B h_1, \dots, \B h_{2m-1} } \sum_{\ol x} e \left( \Delta_{j_1, \B h_1} \cdots \Delta_{j_{2m-1}, \B h_{2m-1}} \F F_0(\ol x; \Ua)\right).
    \end{align*}
    Since every application of the difference operator reduces the degree by one, the argument of the exponential is linear in $\ol x$, and since we had $d<2m$, the dependence on the form $F$ has disappeared up to at most a constant twist. Write for brevity $\cal H = (\B h_1, \dots, \B h_{2m-1})$. In this notation we obtain
    \begin{align*}
        &\sum_{\ol x} e \left( \Delta_{j_1, \B h_1} \cdots \Delta_{j_{2m-1}, \B h_{2m-1}} \F F_0(\ol x; \Ua)\right) \\
        & \ll P^{(m-1)s} \bigg|\sum_{\B x_{j_{2m}}} e\left( \alpha_0 {\det}_{\cal T} \left( (\B h_1, \dots, \B h_m)^t \, (\B h_{m+1}, \dots, \B h_{2m-1}, \B x_{j_{2m}})\right)\right) \bigg|,
    \end{align*}
    and therefore
    \begin{align*}
        |T_P(\Ua)|^{2^{2m-1}} \ll P^{(2^{2m-1}m-2m)s} \sum_{\B h_1, \dots, \B h_{2m-1} } \prod_{n=1}^s \min \left\{ P, \|\alpha_0 d_n(\cal H)\|^{-1} \right\}.
    \end{align*}
    The remainder of the proof follows precisely by the argument of Birch \cite[\S 2]{birch}.
\end{proof}

Birch proceeds by interpreting the third case in Lemma~\ref{L2.2} geometrically. In our setting, however, we encounter a delicacy when embedding the discriminant function into the complex numbers. In fact, the complex embedding of the discriminant is $\det (X^*\, X)$, where $X^*$ denotes the adjoint of the matrix $X$, but since this operation is not polynomial, the complex discriminant function is not amenable to the methods of classical algebraic geometry. It turns out, however, that in our specific case the same ideas underlying the usual arguments from algebraic geometry may still be made to work over the real numbers.

\begin{lem}\label{L2.3}
  Case (C) of Lemma~\ref{L2.2} does not occur when
  \begin{align}\label{2.4}
      s>2^{2m-1}l.
  \end{align}
\end{lem}

\begin{proof}
  The set of all $\B h_1, \dots, \B h_{2m-1} \in \R^s$ satisfying \eqref{2.3} forms a real manifold $\sc M$ inside $\R^{(2m-1)s}$. Furthermore, Lemma~3.1 of Birch \cite{birch} shows that whenever case (C) of Lemma~\ref{L2.2} applies, then one has
  \begin{align}\label{2.5}
    \dim \sc M \geqslant (2m-1)s - 2^{2m-1}l.
  \end{align}
  Observe that for any fixed choice of $\B x_1, \dots, \B x_{m-1}$ and $\B y_1, \dots, \B y_{m-1}$ the polynomial $\det\left( (\B x_1, \dots,  \B x_m)^t \,(\B y_1, \dots, \B   y_{m}) \right)$
  is homogeneous in $\B x_m$ and $\B y_m$, and by Euler's theorem on homogeneous functions one has
  \begin{align*}%\label{2.6}
    \det ( (\B h_1, \dots, \B h_m)^t \, (\B h_{m+1}, \dots, \B h_{2m-1}, \B h_m)) = 0
  \end{align*}
  for all $\B h_1, \dots, \B h_{2m-1}$ satisfying \eqref{2.3}. Observe further that the value of $\det(\ol x^t \, \ol y)$ is invariant under the action of the symmetric group $S_m$ on the indices.
  Let
    $$\sc M_0 = \{ \ol x, \ol y \in \R^{ms}: \det(\ol x^t \, \ol y) =0\},$$
  and for $1 \le k \le m$ define recursively $\sc M_k = \sc M_{k-1} \cap \sc D_k$, where $\sc D_k \subseteq \R^{2ms}$ denotes the diagonal given by $\B x_k = \B y_k$. It is now clear that $\sc M \subseteq \sc M_1$.

  For a given integer $k$ satisfying $1 \le k \le m-1$, suppose that some tuples $\B x_1, \dots, \B x_{m-1}$ and $\B y_{1}, \dots, \B y_{m-1}$ with $\B x_i = \B y_i$ for all $1 \leqslant i \leqslant k$ satisfy
  \begin{align*}
    \dim \langle \B x_1, \dots, \B x_{m-1} \rangle = \dim \langle \B y_{1}, \dots, \B y_{m-1}\rangle = m-1,
  \end{align*}
    then the function $\det(\ol x^t \, \ol y)$ vanishes precisely if either $\B x_m \in \langle \B x_1, \dots, \B x_{m-1} \rangle$ and $\B y_m$ is arbitrary, or if $\B y_m \in \langle \B y_{1}, \dots, \B y_{m-1}\rangle$ and $\B x_m$ is arbitrary, so the equation $\det(\ol x^t \, \ol y)=0$, interpreted as an equation in $\B x_m$ and $\B y_m$, defines an $(m-1+s)$-dimensional manifold inside $\R^{2s}$. Accordingly, the restriction of $\det(\ol x^t \, \ol y)$ to $\sc D_{m}$ vanishes precisely for those vectors $\B x_{m}$ lying in the $(m-1)$-dimensional manifold $ \langle \B x_1, \dots, \B x_{m-1} \rangle \cup \langle \B y_{1}, \dots, \B y_{m-1}\rangle$. Recalling the invariance under $S_m$, we thus obtain the recursive relationship
  \begin{align}\label{2.6}
        \dim \sc M_{k+1} &= \dim \sc M_k-s,
  \end{align}
  and it is clear that after applying \eqref{2.6} iteratively $m-1$ times we obtain
  \begin{align*}
    \dim \sc M \leqslant \dim \sc M_1 &= \dim \sc M_m  + (m-1)s \\
    &= \dim \{ \ol x \in \R^{ms}: D(\ol x)=0\} + (m-1)s = 2(m-1)s.
  \end{align*}
  Under the hypotheses of the lemma this leads to a contradiction with \eqref{2.5}.
\end{proof}

If $\alpha_0$ has an approximation as in case (B) of Lemma~\ref{L2.2}, we may in a second step establish a result similar to that of Lemma~\ref{L2.2} in order to find a rational approximation for $\ba$. An application of Lemma~\ref{L2.1} with $k=d-2$ yields
\begin{align*}
  |T_P(\Ua)|^{2^{d-2}} \ll P^{((2^{d-2}-1)m - (d-2))s} \sum_{\B h_1, \dots, \B h_{d-2} } \sum_{\ol x} e \left( \Delta_{j_1, \B h_1} \cdots \Delta_{j_{d-2}, \B h_{d-2}} \F F_0(\ol x; \Ua)\right),
\end{align*}
and thus with Cauchy's inequality
\begin{align}\label{2.7}
  |T_P(\Ua)|^{2^{d-1}} \ll P^{(2^{d-1}-2)ms - (d-2)s} \sum_{\B h_1, \dots, \B h_{d-2} } \left|\sum_{\ol x} e \left( \Delta_{j_1, \B h_1} \cdots \Delta_{j_{d-2}, \B h_{d-2}} \F F_0(\ol x; \Ua)\right)\right|^2.
\end{align}
We perform a van der Corput step similar to \S 4 in \cite{bhb}. Let $q_0$ be as in Lemma~\ref{L2.2} (B), write $\nu = (2m-1)\eta$, and observe that
\begin{align*}
  & P^{2(1-\nu)s} \left|\sum_{\ol x} e \left( \Delta_{j_1, \B h_1} \cdots \Delta_{j_{d-2}, \B h_{d-2}} \F F_0(\ol x; \Ua)\right) \right|^2 \\
  & \quad \ll \left| \sum_{\ol x} \sum_{|\B u| \ll P^{1-\nu} } e \left( \Delta_{j_1, \B h_1} \cdots \Delta_{j_{d-2}, \B h_{d-2}} \F F_0(\B x_1, \dots, \B x_{j_{d-1}} + q_0 \B u, \dots, \B x_m; \Ua)\right)  \right|^2 \\
  & \quad \ll P^{ms} \sum_{\ol x} \left|\sum_{|\B u| \ll P^{1-\nu} } e \left( \Delta_{j_1, \B h_1} \cdots \Delta_{j_{d-2}, \B h_{d-2}} \F F_0(\B x_1, \dots, \B x_{j_{d-1}} + q_0 \B u, \dots, \B x_m; \Ua)\right)  \right|^2.
\end{align*}

The square expression can be expanded and is then
\begin{align*}
  &\sum_{|\B u|, |\B v| \ll P^{1-\nu} } e \bigg( \Delta_{j_1, \B h_1} \cdots \Delta_{j_{d-2}, \B h_{d-2}} \big( \F F_0(\B x_1, \dots, \B x_{j_{d-1}} + q_0 \B u, \dots, \B x_m; \Ua) \\
  &\qquad \qquad \qquad \qquad \qquad \qquad \qquad \quad - \F F_0(\B x_1, \dots, \B x_{j_{d-1}} + q_0 \B v, \dots, \B x_m; \Ua)\big)\bigg) \\
  & \ll P^{(1-\nu)s}\sum_{|\B w| \ll P^{1-\nu}}  e \left( \Delta_{j_1, \B h_1} \cdots \Delta_{j_{d-2}, \B h_{d-2}} \Delta_{j_{d-1}, q_0 \B w}\F F_0(\ol x; \Ua)\right),
\end{align*}
where we made a change of variables $\B x_{j_{d-1}} \mapsto \B x_{j_{d-1}}+q_0 \B v$ and $\B w = \B u - \B v$.
It follows that
\begin{align}\label{2.8}
    &\left|\sum_{\ol x} e \left( \Delta_{j_1, \B h_1} \cdots \Delta_{j_{d-2}, \B h_{d-2}} \F F_0(\ol x; \Ua)\right)\right|^2 \nonumber\\
    & \quad \ll P^{ms-(1-\nu)s} \sum_{|\B w| \ll P^{1-\nu}} \sum_{\ol x}e \left( \Delta_{j_1, \B h_1} \cdots \Delta_{j_{d-2}, \B h_{d-2}} \Delta_{j_{d-1}, q_0 \B w}\F F_0(\ol x; \Ua)\right).
\end{align}
Observe that
\begin{align}\label{2.9}
     &\Delta_{j_1, \B h_1} \cdots \Delta_{j_{d-2}, \B h_{d-2}} \Delta_{j_{d-1}, q_0 \B w}\F F(\ol x; \Ua)\nonumber \\
     &\quad = q_0 \sum_{k=1}^m M(j_1, \dots, j_{d-1}, k)\alpha_{(j_1, \dots, j_{d-1}, k)} \Phi(\B x_k, \B h_1, \dots, \B h_{d-2}, \B w) \nonumber\\
     &\qquad  + \text{ terms independent of $\ol x$},
\end{align}
where $M(\B j)$ is a combinatorial factor. On the other hand, it follows from the assumption on $\alpha_0$ that
\begin{align*}
  \frac{\D}{\D x_{ij}} e \left( \alpha_0 \Delta_{j_1, \B h_1} \cdots \Delta_{j_{d-2}, \B h_{d-2}} \Delta_{j_{d-1}, q_0 \B w} D(\ol x)\right) \ll \|q_0 \alpha_0 \| P^{2m-1-\nu} \ll P^{-1}
\end{align*}
for all $1 \leqslant i \leqslant m, 1 \leqslant j \leqslant s$, so a multi-dimensional summation by parts shows that the contribution to \eqref{2.8} stemming from $D$ is negligible.

For the sake of notational conciseness write $\cal H$ for the tuple $(\B h_1, \dots, \B h_{d-2})$. Define the $(d-1)$-linear forms $B_n$, $1 \leqslant n \leqslant s$, via
\begin{align*}%\label{2.x}
  \Phi(\B x^{(1)}, \dots, \B x^{(d)}) = \sum_{n=1}^s  x_n^{(d)} B_n (\B x^{(1)}, \dots, \B x^{(d-1)}),
\end{align*}
then a combination of \eqref{2.7}, \eqref{2.8} and \eqref{2.9} together with a familiar bound on linear exponential sums yields
\begin{align*}
  |T_P(\Ua)|^{2^{d-1}} &\ll P^{(2^{d-1}-1)ms - (d-1-\nu)s} P^{(m-1)s}\sum_{\cal H } \sum_{\B w} \bigg|\sum_{\B x_{j_d}} e(q_0  M(\B j)\alpha_{\B j} \Phi(\cal H, \B w,\B x_{j_d}))\bigg|\\
  &\ll P^{(2^{d-1}m - d + \nu)s} \Upsilon_{\B j}
\end{align*}
for every $\B j \in J$, where $\Upsilon_{\B j}$ is given by
\begin{align*}
    \Upsilon_{\B j} = \sum_{\cal H} \sum_{|\B w| \ll P^{1-\nu} } \prod_{n=1}^s \min \{P, \|M(\B j) q_0 \alpha_{\B j} B_n(\cal H, \B w) \|^{-1} \}.
\end{align*}
Define
\begin{align*}
    N_{\B j}(X, Y; Z) = \card \{ & |\B h_1|, \dots, |\B h_{d-2}| \leqslant X, |\B w| \leqslant Y: \\
    &\|M(\B j) q_0 \alpha_{\B j} B_n(\B h_1, \dots, \B h_{d-2}, \B w) \|< Z \; (1 \le n \le s) \},
\end{align*}
then standard arguments from the geometry of numbers (cf. \S 4 in \cite{bhb} or Lemma~3.4 in \cite{FRF2}) show that for any $0 < \theta \leqslant 1-\nu$ one has the estimate
\begin{align*}
  \Upsilon_{\B j} &\ll P^{s + \eps} N_{\B j}(P, P^{1-\nu}; P^{-1})\\
  &\ll P^{s+\eps} P^{(d-1)(1-\theta)s - \nu s} N_{\B j} (P^{\theta}, P^{\theta}; P^{-d+(d-1)\theta + \nu}).
\end{align*}
It follows that, if $|T_P(\Ua)| \gg P^{ms-k\theta}$ for some $k>0$ and $0 < \theta \leqslant 1-\nu$, then
\begin{align*}%\label{2.10}
  N_{\B j}(P^{\theta}, P^{\theta}; P^{-d+(d-1)\theta + (2m-1)\eta}) \gg (P^{\theta})^{(d-1)s - 2^{d-1}k - \eps},
\end{align*}
and we may conclude as follows.

\begin{lem}\label{L2.4}
    Suppose that $\alpha_0$ has an approximation as in case (B) of Lemma~\ref{L2.3} with denominator $q_0$, and let $k>0$ and $0 < \theta \leqslant 1-(2m-1)\eta$ be parameters. Then one of the following is true.
    \begin{enumerate}[(A)]
      \item Either $|T_P(\Ua)| \ll P^{ms-k\theta}$, or
      \item for any $\B j \in J$ one finds $q_{\B j} \ll P^{(d-1)\theta}$ and $0 \leqslant a_{\B j} < q_0 q_{\B j}$ satisfying
	\begin{align*}
	  |q_0 q_{\B j} \alpha_{\B j} - a_{\B j}| \ll P^{-d+(d-1)\theta + (2m-1)\eta},
	\end{align*}
        or
      \item the number of $(d-1)$-tuples $(\B h_1, \dots, \B h_{d-1})$ with $|\B h_i|\ll P^{\theta}$ $(1 \le i \le d-1)$ satisfying
	\begin{align*}
	    B_n(\B h_1, \dots, \B h_{d-1})=0 \qquad (1 \leqslant n \leqslant s)
	\end{align*}
	is asymptotically greater than $(P^{\theta})^{(d-1)s - 2^{d-1}k - \eps}$.
    \end{enumerate}
\end{lem}

This follows as in the proof of Lemma~3.4 in \cite{FRF}, and as in Lemma~3.5 of the same work the third case may be excluded by demanding that
\begin{align}\label{2.10}
    s - \dim \sing F > 2^{d-1}k.
\end{align}
The next step is to combine Lemmata \ref{L2.2} and \ref{L2.4} in order to bound $T_P(\Ua)$ on the minor arcs.

\section{The minor arcs in the case $2m>d$}

The goal of this section is to define our sets of major and minor arcs, and to show that under certain conditions the former set can be taken arbitrarily small. Throughout this section we will always assume the inequalities \eqref{2.4} and \eqref{2.10} to hold.

For non-negative coprime integers $a <q$ denote by $\F M_0(q,a)$ the set of $\alpha \in \T$ satisfying $|q\alpha-a| \leqslant P^{-2m+(2m-1)\eta}$, and define further
\begin{align*}
    \F M_0(P; \eta) = \bigcup_{q=1}^{P^{(2m-1)\eta}} \bigcup_{\substack{a=0\\(a,q)=1}}^{q-1} \F M_0(q,a)
\end{align*}
and $\F m_0(P; \eta) = \T \setminus \F M_0(P; \eta)$, where the parameter $P$ will be suppressed whenever there is no danger of confusion. Lemma~\ref{L2.2} then implies that whenever $\alpha_0 \not\in \F M_0(\eta)$, then we have $|T_P(\Ua)| \ll P^{ms-l\eta+\eps}$. This allows us to establish our first pruning lemma.

\begin{lem}\label{L3.1}
  Suppose that the parameters $l >0$ and $\eta_* \in (0,1]$ satisfy
  \begin{align}\label{3.1}
    l>2m+rd
  \end{align}
  and
  \begin{align}\label{3.2}
    (l-2(2m-1))\eta_* > rd.
  \end{align}
  Then there exists some $\delta>0$ for which one has
  \begin{align*}
    \int_{\F m_0(\eta_*)} |T_P(\Ua)| \D \alpha_0 \ll P^{ms-2m-rd-\delta}
  \end{align*}
  uniformly for all $\ba \in \T^r$.
\end{lem}

\begin{proof}
  Let $l$ and $\eta_*$ be given according to \eqref{3.1} and \eqref{3.2}. We can find a sequence $(\eta_i)$ with the property
  \begin{align}\label{3.3}
    1 = \eta_0 > \eta_1 > \dots > \eta_{T_1} = \eta_*
  \end{align}
  and subject to the condition
  \begin{align}\label{3.4}
    (\eta_{i-1} - \eta_i)l < (l-2(2m-1))\eta_* -rd \qquad (1 \leqslant i \leqslant T_1).
  \end{align}
  This is always possible with $T_1=O(1)$.
  It follows from Lemma~\ref{L2.2} and \eqref{3.1} that for some $\delta>0$ we have
  \begin{align*}
    \int_{\F m_0(\eta_0)} |T_P(\Ua)| \D \alpha_0 \ll P^{ms-2m -rd-\delta}.
  \end{align*}
  Furthermore, a straightforward calculation shows that
  \begin{align}\label{3.5}
    \vol \F M_0(\eta) \ll P^{-2m +2(2m-1)\eta}.
  \end{align}
  We may therefore decompose the remaining set $\F m_0(\eta_*) \setminus \F m_0(\eta_0)$ according to \eqref{3.3}. With \eqref{3.5} and Lemma~ \ref{L2.2} (A), this yields
  \begin{align*}
      \int_{\F m_0(\eta_*) \setminus \F m_0(\eta_0)} |T_P(\Ua)| \D \alpha_0 & = \sum_{i=1}^{T_1} \int_{\F M_0(\eta_{i-1}) \setminus \F M_0(\eta_i)}    |T_P(\Ua)| \D \alpha_0 \\
      & \ll \max_{1 \leqslant i \leqslant T_1} \vol \F M_0(\eta_{i-1}) \sup_{\alpha_0 \in \F m_0(\eta_i)} |T_P(\Ua)|\\
      & \ll \max_{1 \leqslant i \leqslant T_1} P^{-2m+2(2m-1)\eta_{i-1}}P^{ms-l\eta_i+\eps},
  \end{align*}
  and the exponent on the right hand side is
  \begin{align*}
      ms-2m+l(\eta_{i-1}-\eta_i) - (l-2(2m-1))\eta_{i-1} + \eps < ms-2m-rd
  \end{align*}
  by \eqref{3.4}, provided $\eps$ has been taken sufficiently small. This proves the lemma.
\end{proof}

The second pruning step involves the major arcs associated to the vector $\ba$. Write $\F M(P; \theta, \eta)$ for the set of $\Ua \in \T^{r+1}$ with $\alpha_0 =a_0/q_0 + \beta_0 \in \F M_0(\eta)$ for which there exist entrywise coprime integer vectors $\B a, \B q$ satisfying
\begin{align*}
    |\alpha_{\B j} q_0 q_{\B j} - a_{\B j}| \leqslant P^{-d + (d-1)\theta + (2m-1)\eta} \quad \text{and} \quad q_{\B j} \leqslant P^{(d-1)\theta} \qquad (\B j \in J),
\end{align*}
and let $\F m(P; \theta, \eta) = \T^{r+1} \setminus \F M(P; \theta, \eta)$, where the parameter $P$ will typically be suppressed. Again, Lemma~\ref{L2.4} implies that we have $|T_P(\Ua)| \ll P^{ms-k\theta+\eps}$ for all $\Ua \in \F m(\theta, \eta)$ having $\alpha_0 \in \F M_0(\eta)$.

It is desirable to have a unique parameter for measuring the size of $T_P(\Ua)$ on the minor arcs, and in fact it will transpire in the course of the argument that no generality is lost by setting
\begin{align}\label{3.6}
  k \theta = l \eta.
\end{align}
Thus the width of $\F M(\theta, \eta)$ can be measured in terms of $\theta$ alone, and we will suppress the redundant parameter $\eta$ in the future. With this convention we have $|T_P(\Ua)| \ll P^{ms-k\theta+\eps}$ for all $\Ua \in \F m(\theta)$, so it respects both the respective case distinctions of Lemmata \ref{L2.2} and \ref{L2.4} simultaneously.

\begin{lem}\label{L3.2}
  Let $\eta_*$ be the parameter obtained in Lemma~\ref{L3.1}, and suppose that the conditions
  \begin{align}\label{3.7}
    (2m-1+l/k)\eta_* \leqslant 1
  \end{align}
   and
   \begin{align}\label{3.8}
      \frac{2r(d-1)}{k} + \frac{(r+2)(2m-1)}{l} <1
   \end{align}
   are satisfied. Then for any $\theta \in (0, 1-(2m-1)\eta_*]$ there exists a $\delta>0$ such that
   \begin{align*}
      \int_{\F m(\theta)} |T_P(\Ua)| \D \Ua \ll P^{ms-rd-2m-\delta}.
   \end{align*}
\end{lem}

\begin{proof}
  A standard computation shows that
  \begin{align*}
    \vol \F M(\theta) &\ll \sum_{q_0=1}^{P^{(2m-1)\eta}} \sum_{a_0=0}^{q_0-1} \frac{P^{-2m+(2m-1)\eta}}{q_0} \prod_{\B j \in J} \left(\sum_{q_{\B j}=1}^{P^{(d-1)\theta}} \sum_{a_{\B j}=0}^{q_0q_{\B j}-1} \frac{P^{-d+(d-1)\theta + (2m-1)\eta}}{q_{\B j}q_0}\right) \\
     &\ll P^{-2m-rd+2r(d-1)\theta + (r+2)(2m-1)(k/l) \theta},
  \end{align*}
  where we used \eqref{3.6}. We now fix a sequence $(\theta_i)$ satisfying
  \begin{align*}
      (l/k)\eta_* = \theta_* =\theta_0 > \theta_1 > \dots > \theta_{T_2} = \theta > 0,
  \end{align*}
  and having the property that
  \begin{align}\label{3.9}
      k(\theta_{i-1}-\theta_i) < (k-2r(d-1)-(r+2)(2m-1)(k/l)) \theta
  \end{align}
  for each $i$. This is possible by \eqref{3.8}, and \eqref{3.7} ensures via \eqref{3.6} that Lemma~\ref{L2.4} is applicable. In fact, from the definition of $\F m(\theta)$ above we have
  \begin{align*}
    \int_{\F m(\theta_*)} |T_P(\Ua)| \D \Ua
    & \ll \int_{\T^r} \int_{\F m_0(\eta_*)} |T_P(\Ua)| \D \alpha_0 \D \ba + \vol \F M_0(\eta_*) \sup_{\Ua \in \F m(\theta_*)}|T_P(\Ua)|,
  \end{align*}
  where the first term is $O(P^{ms-rd-2m-\delta})$ by Lemma~\ref{L3.1} and the second term can be bounded above by $P^{-2m+2(2m-1)\eta_*}P^{ms-k\theta_*+\eps}$. On recalling \eqref{3.6} and \eqref{3.2} we see that for sufficiently small $\eps$ the exponent is smaller than $ms-rd-2m$.

  We now argue as before and find that
  \begin{align*}
    \int_{\F m(\theta) \setminus \F m(\theta_*)}|T_P(\Ua)| \D \Ua &\ll \max_{1 \leqslant i \leqslant T_2}\vol \F M(\theta_{i-1}) \sup_{\Ua \in \F m(\theta_i)} |T_P(\Ua)|\\
    & \ll \max_{1 \leqslant i \leqslant T_2}  P^{-2m-rd+(2r(d-1) + (r+2)(2m-1)(k/l)) \theta_{i-1}}P^{ms-k\theta_i+\eps},
  \end{align*}
  and the exponent is
  \begin{align*}
    ms-rd-2m+k(\theta_{i-1}-\theta_i) - \left(k - 2r(d-1)-(r+2)(2m-1)(k/l) \right)\theta_{i-1}+\eps,
  \end{align*}
  which is smaller than $ms-rd-2m$ by \eqref{3.9} whenever $\eps$ is sufficiently small.
\end{proof}

Observe that Lemmata \ref{L3.1} and \ref{L3.2} are compatible only if one can find a value for $\eta_*$ satisfying both \eqref{3.2} and \eqref{3.7}. This is possible if and only if
\begin{align}\label{3.10}
 \frac{rd}{k} + \frac{(2+rd)(2m-1)}{l} <1.
\end{align}
This condition fully encompasses \eqref{3.1}. It follows that, if the conditions \eqref{3.8} and \eqref{3.10} are satisfied, we may choose $\theta$ (and thereby also $\eta$) arbitrarily small.

\section{The major arcs bound for $2m>d$}

For the analysis of the contribution from our narrow set of major arcs it is desirable to have approximations of the components of $\Ua$ that use the same denominator. We therefore set
\begin{align*}
    q= \lcm (q_0, \B  q) \ll P^{(r(d-1) + (2m-1)(k/l))\theta}.
\end{align*}
Write $(r(d-1) + (2m-1)(k/l))\theta= \omega$, and for fixed $q, \U a$ let $\F N(q,\UA)$ denote the set of all $\Ua \in \T^{r+1}$ satisfying
\begin{align*}
    |\alpha_0 - a_0/q| \leqslant  P^{-2m+\omega} \quad \text{ and } \quad |\alpha_{\B j} - a_{\B j}/q| \leqslant  P^{-d+\omega} \qquad (\B j \in J).
\end{align*}
Define further
\begin{align*}
    \F N(\theta) = \bigcup_{q=1}^{P^{\omega}} \bigcup_{\substack{\U a=0 \\ (\U a,q)=1}}^{q-1} \F N(q, \U a).
\end{align*}
Then $\F M(\theta) \subseteq \F N(\theta)$. One computes
\begin{align}\label{4.1}
  \vol \F N(\theta) \ll P^{-2m-rd+(2r+3)\omega}.
\end{align}
Define
\begin{align}\label{4.2}
  S_q(\U a) = \sum_{\ol x = 1}^q e(\F F_0(\ol x; q^{-1} \U a))
\end{align}
and
\begin{align} \label{4.3}
  v_P(\Ub) = \int_{[-P, P]^{ms}} e(\F F_0(\ol{\bm \xi}; \Ub)) \D \ol{\bm \xi},
\end{align}
then a standard argument reveals that
\begin{align}\label{4.4}
  |T_P(\Ua)- q^{-ms}S_q(\U a) v_P(\Ub)| \ll P^{ms-1}q \left( \sum_{\B j \in J} |\beta_{\B j}|P^d + |\beta_0|P^{2m}  + 1\right).
\end{align}
Write further
\begin{align}\label{4.5}
    \F S_b(P) = \sum_{q=1}^{P^\omega} q^{- ms} \sum_{\substack{\U a=0\\(\U a,q)=1}}^{q-1} S_q(\U a) e(-a_0b/q)
\end{align}
and
\begin{align}\label{4.6}
 \F J_b(P) = \int_{|\bb| \leqslant P^{-d+\omega}} \int_{|\beta_0| \leqslant P^{-2m+\omega}} v_P(\Ub) e(-b\beta_0) \D \Ub
\end{align}
for the truncated singular series and integral. It then follows from \eqref{4.4} and \eqref{4.1} that
\begin{align}\label{4.7}
  \int_{\F N(\theta)} T_P(\Ua) e(-\alpha_0b) \D \Ua = \F S_b(P) \F J_b(P) + O(P^{ms-rd-2m+(2r+5)\omega - 1}),
\end{align}
where the error is acceptable if $\theta$ has been chosen small enough. By a change of variables one has
\begin{align}\label{4.8}
    v_P(\Ub) = P^{ms}v_1(P^d\bb, P^{2m }\beta_0),
\end{align}
and thus
\begin{align}\label{4.9}
    \F J_b(P)= P^{ms-rd-2m} \int_{| \Ub | \leqslant P^{\omega}} v_1(\Ub) e(b \beta_0 /P^{2m}) \D \Ub.
\end{align}
It remains to show that $\F S_b(P)$ and the integral in the expression for $\F J_b(P)$ converge as $P \to \infty$ and reproduce the expected main term.

\begin{lem}\label{L4.1}
    The terms of the singular series are bounded by
    \begin{align*}
        |q^{-ms}S_q(\U a)| \ll q^{\eps}\min \left\{\left(\frac{q}{(q,a_0)} \right)^{\frac{l}{2m-1}}, q^{\left(\frac{2m-1}{l} + \frac{d-1}{k} \right)^{-1}}\right\}.
    \end{align*}
\end{lem}

\begin{proof}
    We imitate the argument of Browning and Heath-Brown \cite[Lemma~8.2]{bhb}. Observe that the statement of the lemma is satisfied when $q=1$, so we can, without loss of generality, suppose that $q >1$. Equally, if $a_0=0$, the first term in the minimum returns the trivial bound, allowing us to assume that $a_0>0$ and therefore $(q,a_0)<q$. Fix $Q = q^A$ for some large $A$ to be determined later. Applying \eqref{4.4} and \eqref{4.8} with $\Ub = \UU 0$ and observing that $v_1(\UU 0) \asymp 1$, it follows that
    \begin{align}\label{4.10}
    		q^{-ms}S_q(\U a) \ll Q^{-ms} |T_Q(q^{-1} \U a)| + q/Q.
    \end{align}
    Fix $\eta$ such that
    \begin{align}\label{4.11}
    		q/(q,a_0) =  Q^{(2m-1)\eta},
    \end{align}
    so that $a_0/q \in \F M_0(Q; \eta)$. Note that by taking $A$ sufficiently large we may ensure $\eta<m/(2m-1)$, so that these major arcs are disjoint. It follows that $a_0/q$ is best approximated by itself. Furthermore, in the $q$-aspect it lies just on the edge of the major arcs. Since by continuity the minor arcs bound applies on the closure of the minor arcs, we have additionally that $|T_Q(q^{-1} \U a)|\ll Q^{ms-l \eta+\eps}$. Solving this for $Q^\eta$ and inserting into \eqref{4.11} yields after rearranging $|T_Q (q^{-1} \U a)| \ll Q^{ms+\eps} \left(q/(q,a_0)\right)^{-\frac{l}{2m-1}}$, and on substituting this into \eqref{4.10}, one sees that
    \begin{align*}
    		q^{-ms}S_q(\U a) \ll Q^{\eps}\left(\frac{q}{(q,a_0)}\right)^{-\frac{l}{2m-1}} + q/Q.
    \end{align*}
    Recalling that $Q=q^A$, it is clear that for $A$ sufficiently large the first term dominates. This establishes the first bound in the lemma.

  	Fix now $\theta$ via
  	\begin{align}\label{4.12}
  			q = Q^{((2m-1)(k/l) + (d-1))\theta},
  	\end{align}
  	so that $q^{-1} \U a \in \F M(Q; \theta)$.
    As before, we are free to take $A$ large enough that the major arcs $\F M(q, \U a)$ are disjoint, and we deduce that in the $q$-aspect, $q^{-1} \U a$ lies on the boundary of $\F M(Q;\theta)$, so the minor arcs estimate of Lemma~\ref{L2.4} (A) still applies and yields $Q^\theta \ll (Q^{-ms+\eps}|T_Q(q^{-1} \U a)|)^{-1/k}$. Together with \eqref{4.12} this produces a non-trivial bound on $T_Q(q^{-1} \U a)$ which in turn may be inserted into \eqref{4.10}, yielding
    \begin{align*}
    		q^{-ms}S_q(\U a) \ll Q^{\eps}q^{-\left(\frac{d-1}{k} + \frac{2m-1}{l}\right)^{-1}} + q/Q.
    \end{align*}
    As before, we see that for $A$ large enough the first term dominates. This establishes the second statement of the lemma.
\end{proof}

Lemma~\ref{L4.1} implies that the singular series may be extended to infinity. In fact, we have
\begin{align*}
	\sum_{q=1}^{\infty} \sum_{\substack{\U a=0 \\ (\U a,q)=1}}^{q-1} q^{-ms} S_q(\U a)
    & \ll \sum_{q=1}^{\infty} q^{r-(1-\lambda)\left(\frac{2m-1}{l} + \frac{d-1}{k} \right)^{-1}+\eps} \sum_{d | q} (q/d)^{1-\lambda\left(\frac{2m-1}{l}\right)^{-1}}
\end{align*}
for each $\lambda \in [0,1]$. This series converges if, for some $\lambda$, one has
\begin{align*}%\label{4.13}
    \frac{2m-1}{l}< \lambda \quad \text{ and } \quad \frac{(r+1)(2m-1)}{l} + \frac{(r+1)(d-1)}{k} < 1-\lambda,
\end{align*}
and these inequalities can be simultaneously satisfied if and only if
\begin{align}\label{4.13}
    \frac{(2m-1)(r+2)}{l} + \frac{(d-1)(r+1)}{k} <1.
\end{align}

It remains to complete the singular integral
\begin{align}\label{4.14}
    \chi_{\infty}(b, P, R) &= \int_{[-R, R]^{r+1}} v_1(\Ub) e(-b \beta_0/P^{2m}) \D \Ub.
\end{align}
This follows the argument of \cite[Lemma~8.3]{bhb}.

\begin{lem}\label{L4.2}
    We have
    \begin{align*}
        |v_1(\Ub)| \ll \min\left\{1, |\beta_0|^{-\frac{l}{2m-1}+ \eps}, |\bb|^{-\left(\frac{2m-1}{l} + \frac{d-1}{k}\right)^{-1}+\eps}\right\}.
    \end{align*}
\end{lem}

\begin{proof}
    We start by observing that the bound $|v_1(\Ub)| \ll 1$ is trivial, so in what follows we do not lose any generality by assuming that $\beta_0 \neq 0$ and $\bb \neq \bm 0$.
    Choose $Q=|\Ub|^A$ for some large parameter $A$ to be fixed later, and write $\Ug = (Q^{-d} \bb, Q^{-2m}\beta_0)$. Taking $\U a=\UU 0$ and $q=1$, we have from \eqref{4.4} and \eqref{4.8} that
    \begin{align}\label{4.15}
	   |v_1(\Ub)| = Q^{-ms}|v_Q(\Ug)| \ll Q^{-ms} |T_Q(\Ug)| + Q^{-1}|\Ub|.
    \end{align}
    Determine $\eta$ such that $|\beta_0| = Q^{(2m-1)\eta}$. Observe that for sufficiently large $A$ one has $\eta < m/(2m-1)$, so we can assume that the major arcs are disjoint. Hence $\gamma_0$ is best approximated by $q_0=1$ and $a_0=0$, and thus lies just on the edge of the major arcs $\F M_0(Q;\eta)$. By continuity, the minor arcs estimate extends to the closure of the minor arcs, so we have $Q^{\eta} \ll (Q^{-ms-\eps}|T_Q(\Ug)|)^{-1/l}$. On the other hand, exploiting the major arcs information about $\gamma_0$, we obtain
    \begin{align*}
	   |\beta_0| \ll Q^{(2m-1)\eta} \ll  \left(Q^{-ms-\eps} |T_Q(\Ug)|\right)^{-\frac{2m-1}{l}}.
    \end{align*}
    Solving this for $ |T_Q(\Ug)|$ and inserting into \eqref{4.15} yields
    \begin{align*}
	   |v_1(\Ub)| \ll Q^\eps|\beta_0|^{-\frac{l}{2m-1}} + Q^{-1}|\Ub|.
    \end{align*}
    Recalling that $Q=|\Ub|^{A}$, this yields the first estimate whenever $A$ is large enough.

    For the second estimate, we fix $\theta$ such that, recalling \eqref{3.6}, we have
    \begin{align*}
	   \max \{Q^{-(d-1)\theta - (2m-1)\eta} |\bb|, Q^{-(2m-1)\eta}|\beta_0|\} = 1.
    \end{align*}
    As in the previous lemma, by choosing $A$ large enough, we may assume that the major arcs are disjoint. The unique best approximation to $\Ug$ is therefore given by $\U a = \UU 0$ and $\U q = \UU 1$. In particular, $\Ug$ lies on the boundary of $\F M(Q; \theta)$. Again, by extending the minor arcs estimate from Lemma~\ref{L2.4} (A) to the boundary, we deduce that $Q^{\theta} \ll \left(Q^{-ms-\eps}|T_Q(\Ug)|\right)^{-1/k}$. On the other hand, our choice of $\theta$ implies that
    \begin{align*}
        |\beta_{\B j}|\ll Q^{((2m-1)(k/l)+(d-1))\theta} \ll \left(Q^{-ms-\eps}|T_Q(\Ug)|\right)^{-\left(\frac{2m-1}{l}+\frac{d-1}{k}\right)}
    \end{align*}
    for every $\B j \in J$. This inequality is easily rearranged to yield a bound on $|T_Q(\Ug)|$, and as before, it follows that
    \begin{align*}
	   v_1(\Ub) \ll Q^\eps|\beta_{\B j}|^{-\left(\frac{2m-1}{l}+\frac{d-1}{k}\right)^{-1}} + |\Ub|^{1-A},
    \end{align*}
    which returns the desired bound if $A$ is large enough.
\end{proof}

Now write $\rho_0 = |\beta_0|$ and $\rho = |\bb|$, and note that the set of $\bb$ satisfying $|\bb| = \rho$ has measure $O(\rho^{r-1})$. Thus, the expression from \eqref{4.14} is bounded above by
\begin{align*}
    \chi_{\infty}(b, P, R) &\ll \int_{0}^R \int_{0}^R \min \Big\{1, \rho_0^{-\frac{l}{2m-1}+\eps}, \rho^{-\left(\frac{2m-1}{l} + \frac{d-1}{k} \right)^{-1}+\eps} \Big\} \rho^{r-1}\D \rho \D \rho_0 \\
    & \ll \left(1+\int_{1}^R  \rho_0^{- \lambda \left(\frac{2m-1}{l}\right)^{-1}+\eps}\D \rho_0 \right) \left( 1+\int_{1}^R \rho^{-(1-\lambda)\left(\frac{2m-1}{l} + \frac{d-1}{k} \right)^{-1}+r-1+\eps} \D \rho     \right)
\end{align*}
for any $\lambda \in [0,1]$. As in the situation regarding the singular series, the limit $\chi_{\infty}(b,P) = \lim_{R \to \infty} \chi_{\infty}(b,P,R)$ exists if the inequalities
\begin{align*}
    \frac{2m-1}{l}< \lambda \quad \text{ and } \quad \frac{r(2m-1)}{l} + \frac{r(d-1)}{k} < 1-\lambda
\end{align*}
can simultaneously be satisfied, which is possible if and only if
\begin{align}\label{4.16}
    \frac{(2m-1)(r+1)}{l} + \frac{(d-1)r}{k} <1.
\end{align}
Both \eqref{4.13} and \eqref{4.16} are a consequence of \eqref{3.8}, so on combining our estimates we obtain
\begin{align}\label{4.17}
  N_m(P;b) = P^{ms-rd-2m} \F S_b \chi_{\infty}(b, P) + O(P^{ms-rd-2m-\delta}),
\end{align}
provided that the conditions \eqref{2.4}, \eqref{2.10}, \eqref{3.8} and \eqref{3.10} are all satisfied. In fact, we may restate the case $2m>d$ of Theorem~\ref{mainthm} in a more general fashion.
\begin{thm}\label{T4.1}
    Let $F$, $m$ and $d$ be as in Theorem~\ref{mainthm} with $2m>d$, and suppose that the conditions
    \begin{align*}
        \frac{2^d r (d-1)}{s - \dim \sing F} + \frac{2^{2m-1}(r+2)(2m-1)}{s} &<1 \quad \text{and} \\
        \frac{2^{d-1}rd}{s - \dim \sing F} + \frac{2^{2m-1}(2+rd)(2m-1)}{s} &<1
    \end{align*}
    are both satisfied. Then for some $\delta>0$ one has
    \begin{align*}
        N_m(P;b) = P^{ms-rd-2m}  \chi_{\infty}(b, P) \prod_{p \text{ prime}} \chi_p(b) + O(P^{ms-rd-2m-\delta}),
    \end{align*}
    where the factors are given by
    \begin{align*}
        \chi_{\infty}(b, P) &= \int_{\R^{r+1}} v_1(\Ub) e(-b \beta_0/P^{2m}) \D \Ub
    \end{align*}
    and
    \begin{align*}
        \chi_p(b) = \lim_{i \to \infty} p^{-ims} \sum_{\ol x = 1}^{p^i} \sum_{\U a = 0}^{p^i-1} e( \F F_0(\ol x; p^{-i}\U a) -p^{-i} b a_0).
    \end{align*}
\end{thm}
The only thing that still remains to be shown is that one has indeed an Euler product representation of the singular series as advertised. This is, however, standard and follows from arguments analogous to those given in Chapter 5 of Davenport's book \cite{dav}. We also remark that the second statement of Theorem~\ref{mainthm} follows upon assuming $\dim \sing F = 0$ and observing that under this assumption the hypotheses of Theorem~\ref{T4.1} reduce to
\begin{align*}
    s > \max\{ 2^d r (d-1)+2^{2m-1}(r+2)(2m-1), 2^{d-1}rd +2^{2m-1}(2+rd)(2m-1)\}.
\end{align*}
A modicum of computation confirms that for $2m>d$ the second term dominates.

\section{Weyl differencing in the case $d>2m$}

In our second case, the procedure is structurally very similar to the treatment of the case $d<2m$. The following is a straightforward modification of Lemma~5.3 of \cite{FRF}.
\begin{lem}\label{L5.1}
    Suppose that $k$ satisfies \eqref{2.10} and we have
    \begin{align}\label{5.1}
        0 < \theta < \frac{d}{(d-1)(r+3)}.
    \end{align}
    Then for $\Ua \in \T^{r+1}$ one of the following holds.
    \begin{enumerate}[(A)]
        \item We have $|T_P(\Ua)| \ll P^{ms-k\theta}$, or
        \item there are integers $1 \leqslant \tq \ll P^{2(d-1)\theta}$ and $0 \leqslant a_{\B j} < \tq$ $(\B j \in J)$ such that
            \begin{align*}
                |\tq \alpha_{\B j} - a_{\B j}| \ll P^{-d+3(d-1)\theta}.
            \end{align*}
    \end{enumerate}
\end{lem}
\begin{proof}
    This follows by the same proof as in \cite[Lemma~5.3]{FRF}. Observe that, since the degree of $D$ is strictly smaller than that of $F$, all the terms involving $D$ disappear in the course of the proof.
\end{proof}

Let now $\theta$ and $\tq$ be fixed, suppose that $\ba$ satisfies the condition of Lemma~\ref{L5.1} (B), and write $\nu = 3(d-1)\theta$. Recall the definition of the discrete differencing operator from \eqref{2.1}, then Lemma~\ref{L2.1} implies that
\begin{align*}
    |T_P(\Ua)|^{2^{2m-2}} \ll P^{(2^{2m-2}m-3m+2)s} \sum_{\B h_1, \dots, \B h_{2m-2}} \sum_{\ol x} e\left(\Delta_{j_1, \B h_1} \cdots \Delta_{j_{2m-2}, \B h_{2m-2}} \F F_0(\ol x; \Ua)\right),
\end{align*}
where the variables $\ol x$ and $\B h_1, \dots, \B h_{2m-2}$ run over boxes contained in $[-P,P]^s$. By Cauchy's inequality we have therefore
\begin{align*}
    |T_P(\Ua)|^{2^{2m-1}} \ll P^{(2^{2m-1}m - 4m + 2)s} \sum_{\B h_1, \dots, \B h_{2m-2}}  \left|\sum_{\ol x} e \left( \Delta_{j_1, \B h_1} \cdots \Delta_{j_{2m-2}, \B h_{2m-2}} \F F_0(\ol x; \Ua)\right) \right|^{2}.
\end{align*}

We abbreviate $\cal H = (\B h_1, \dots, \B h_{2m-2})$. By an argument mirroring the treatment of the case $2m>d$ leading to \eqref{2.8}, we observe that
\begin{align*}
    & \left|\sum_{\ol x} e \left( \Delta_{j_1, \B h_1} \cdots \Delta_{j_{2m-2}, \B h_{2m-2}} \F F_0(\ol x; \Ua)\right) \right|^2 \\
    & \quad \ll P^{ms-(1-\nu)s} \sum_{\ol x} \sum_{|\B w| \ll P^{1-\nu}} e \left( \Delta_{j_1, \B h_1} \cdots \Delta_{j_{2m-2}, \B h_{2m-2}} \Delta_{j_{2m-1}, \tq \B w}  \F F_0(\ol x; \Ua)\right),
\end{align*}
whence we conclude that
\begin{align}\label{5.2}
    |T_P(\Ua)|^{2^{2m-1}} \ll P^{(2^{2m-1}m - 3m + 1 + \nu)s} \sum_{\cal H} \sum_{\B w} \sum_{\ol x} e \left( \Delta_{j_1, \B h_1} \cdots  \Delta_{j_{2m-1}, \tq \B w}  \F F_0(\ol x; \Ua)\right).
\end{align}
Similar to before, we observe that for all $\B j \in J$ and all $1 \leqslant i \leqslant m$, $1 \leqslant n \leqslant s$ one has
\begin{align*}
    \frac{\D}{\D x_{i,n}} e \left( \Delta_{j_1, \B h_1} \cdots \Delta_{j_{2m-2}, \B h_{2m-2}} \Delta_{j_{2m-1}, \tq \B w}  \F F(\ol x; \Ua)\right) \ll \|\tq \alpha_{\B j} \| P^{d-1-\nu} \ll P^{-1}
\end{align*}
from our assumption on $\ba$, so it follows from partial summation that the dominating contribution in \eqref{5.2} stems from $D(\ol x)$.
Recall our notation \eqref{2.2}, then we find
\begin{align}\label{5.3}
    |T_P(\Ua)|^{2^{2m-1}} &\ll P^{(2^{2m-1}m - 3m+1 + \nu)s}  \sum_{\cal H} \sum_{\B w} \left|\sum_{\ol x}  e \left( \alpha_0 \Delta_{j_1, \B h_1} \cdots  \Delta_{j_{2m-1}, \tq \B w}  D(\ol x)\right) \right| \nonumber \\
    & \ll  P^{(2^{2m-1}m - 2m + \nu)s}  \sum_{\cal H} \sum_{\B w} \prod_{n=1}^s \min \left\{P, \|  \tq \alpha_0d_n(\cal H, \B w) \|^{-1} \right\}.
\end{align}
Let
\begin{align*}
  \Upsilon = \sum_{\cal H} \sum_{| \B w| \ll P^{1-\nu}} \prod_{n=1}^s \min \left\{P, \|  \tq \alpha_0 d_n(\cal H, \B w) \|^{-1} \right\},
\end{align*}
and define
\begin{align*}
    N(X, Y; Z) = \card \{  | \B h_1|, \dots, |\B h_{2m-2}| \leqslant X, |\B w| \leqslant Y: \| \tq \alpha_0 d_n(\cal H, \B w) \| < Z \},
\end{align*}
then arguments from the geometry of numbers (see \cite[\S 4]{bhb} or \cite[Lemma~3.4]{FRF2}) show that for every $\eta \in (0, 1-\nu]$ one has
\begin{align}\label{5.4}
    \Upsilon &\ll P^{s + \eps} N(P, P^{1-\nu}; P^{-1}) \nonumber \\
    & \ll P^{s + \eps} P^{(2m-1)(1-\eta)s - \nu s} N(P^{\eta}, P^{\eta}; P^{-2m + (2m-1)\eta + \nu}).
\end{align}
Suppose now that $| T_P(\Ua)| \gg P^{ms- l \eta}$, then substituting \eqref{5.4} into \eqref{5.3} yields
\begin{align*}
    N(P^{\eta}, P^{\eta}; P^{-2m + (2m-1)\eta + \nu})\gg P^{(2m-1)s\eta - 2^{2m-1}l \eta - \eps},
\end{align*}
and as before, the argument of the proof of Lemma~3.4 of \cite{FRF} leads us to the following Weyl type dissection.

\begin{lem}\label{L5.2}
    Suppose that $\tq$ and $\theta$ are as in Lemma~\ref{L5.1}, and let $l$ and $\eta$ be fixed positive parameters satisfying $0 < \eta \le 1-3(d-1)\theta$. Then for every $\Ua \in \T^{r+1}$ one of the following holds.
    \begin{enumerate}[(A)]
        \item We have $|T_P(\Ua)| \ll P^{ms-l\eta}$, or
        \item there are integers $1 \leqslant q_0 \ll P^{(2m-1) \eta}$ and $1 \leqslant a_0 < q_0 \tq$ satisfying
        \begin{align*}
            |\alpha_0 \tq q_0 - a_0| \ll P^{-2m + (2m-1)\eta + 3(d-1)\theta},
        \end{align*}
        or
        \item the number of integral $\B h_1, \dots, \B h_{2m-1} \in  [-P^{\eta}, P^{\eta}]^s$ satisfying
        \begin{align*}
            d_n(\B h_1, \dots, \B h_{2m-1}) = 0 \qquad (1 \leqslant n \leqslant s )
        \end{align*}
        is asymptotically larger than $(P^{\eta})^{(2m-1)s - 2^{2m-1}l - \eps}$.
    \end{enumerate}
\end{lem}
Lemma~\ref{L2.3} above allows us to exclude the third case by demanding that \eqref{2.4} holds.
As before, under certain conditions we may combine Lemmata \ref{L5.1} and \ref{L5.2} to show that on a large set of minor arcs the contribution is smaller than the expected main term.

\section{The minor arcs in the case $d>2m$}

Throughout this section we make the assumptions \eqref{2.4} and \eqref{2.10}. The treatment of the minor arcs is similar to that of \S 3. However, without further measures the constraint imposed upon $\theta$ in Lemma~\ref{L5.1} would lead to unnecessarily large bounds. Fortunately, this can be avoided by pruning instead a different set of major arcs that can be defined for any positive $\theta \le 1$. We record here Lemma~3.5 of \cite{FRF}, which serves as starting point for our first pruning step.

\begin{lem}\label{L6.1}
    Let $\theta \in (0,1]$ and $k>0$ be parameters, where $k$ satisfies \eqref{2.10}. Then one of the following is true.
    \begin{enumerate}[(A)]
        \item We have $|T_P(\Ua)| \ll P^{ms-k\theta}$, or
        \item for each $\B j \in J$ there are integers $0 \leqslant a_{\B j} < q_{\B j} \ll P^{(d-1)\theta}$ satisfying $|\alpha_{\B j} q_{\B j}-a_{\B j}| \ll P^{-d + (d-1)\theta}$.
    \end{enumerate}
\end{lem}
As in the case of Lemma~\ref{L5.1}, the contribution of $D$ disappears in the course of the proof as $D$ is of strictly smaller degree than $F$.

Denote by $\F M_d(P;\theta)$ the set of $\ba \in \T^r$ with the property that one can find entrywise coprime vectors $0 \leqslant \B a < \B q \leqslant P^{(d-1)\theta}$ satisfying $|\alpha_{\B j} q_{\B j}-a_{\B j}| \leqslant P^{-d+(d-1)\theta}$ for each $\B j \in J$. Write further $\F m_d(P; \theta) = \T^{r} \setminus \F M_d(P; \theta)$, then Lemma~\ref{L6.1} shows that we have $|T_P(\Ua)| \ll P^{ms-k\theta+\eps}$ for all $\ba \in \F m_d(P;\theta)$. As usual, we will suppress the parameter $P$ in most cases.

\begin{lem} \label{L6.2}
    Suppose that $k>0$ and $\theta_* \in (0,1]$ satisfy
    \begin{align}\label{6.1}
        k>dr+2m
    \end{align}
    and
    \begin{align}\label{6.2}
        (k-2r(d-1))\theta_*>2m.
    \end{align}
    Then there exists a parameter $\delta>0$ such that uniformly for all $\alpha_0 \in \T$ one has
    \begin{align*}
        \int_{\F m_d(\theta_*)} |T_P(\Ua)| \D \ba \ll P^{ms-rd-2m-\delta}.
    \end{align*}
\end{lem}

\begin{proof}
    Fix a sequence $(\theta_i)$ with $T_3=O(1)$ terms satisfying
    \begin{align*}
        1= \theta_0 > \theta_1 > \dots > \theta_{T_3} = \theta_*
    \end{align*}
    and having the property
    \begin{align}\label{6.3}
        k(\theta_{i-1}- \theta_i) < (k-2r(d-1))\theta_* - 2m \qquad (1 \leqslant i \leqslant T_3).
    \end{align}
    From \eqref{6.1} we infer that there exists a $\delta>0$ such that
    \begin{align*}
        \int_{\F m_d(\theta_0)} |T_P(\Ua)| \D \ba \ll P^{ms-rd-2m-\delta}.
    \end{align*}
    Furthermore, one computes
    \begin{align}\label{6.4}
        \vol \F M_d(\theta) \ll P^{-rd + 2r(d-1)\theta}
    \end{align}
    (see  e.g. equation (4.2) in \cite{FRF}), so on the difference set one has
    \begin{align*}
        \int_{\F m_d(\theta_*) \setminus \F m_d(\theta_0)} |T_P(\Ua)| \D \ba & \ll \max_{1 \leqslant i \leqslant T_3} \vol \F M_d(\theta_{i-1}) \sup_{\ba \in \F m_d(\theta_i)} |T_P(\Ua)| \\
        & \ll \max_{1 \leqslant i \leqslant T_3} P^{-rd + 2r(d-1)\theta_{i-1}}P^{ms-k \theta_i+\eps}
    \end{align*}
    by \eqref{6.4} and Lemma~\ref{L6.1} (A), and \eqref{6.3} ensures that the exponent is smaller than $ms-rd-2m$ whenever $\eps$ is small enough.
\end{proof}

This set $\F M_d(\theta)$ of major arcs has inhomogeneous denominators, so in order to be able to define major arcs for $\alpha_0$ as well we first need to define a second set of homogenised major arcs. Suppose that $\theta^{\dagger}$ is small enough so that \eqref{5.1} holds, then we define $\F M_d^{\dagger}(\tilde q, \B a)$ to be the set of all $\ba \in \T^r$ satisfying $|\alpha_{\B j} \tilde q - a_{\B j}| \leqslant P^{-d+3(d-1)\theta}$, and
\begin{align*}
  \F M_d^{\dagger}(P; \theta) = \bigcup_{\tilde q=1}^{P^{2(d-1)\theta}} \bigcup_{\substack{\B a = 0 \\ (\B a, \tilde q)=1}}^{\tilde q-1} \F M_d^{\dagger}(\tilde q, \B a).
\end{align*}
Again, we let $\F m_d^{\dagger}(\theta) = \T^r \setminus \F M_d^{\dagger}(\theta)$ and note that this dissection into major and minor arcs respects the case distinction of Lemma~\ref{L5.1}. Lemma~5.3 of \cite{FRF} shows that $\F M_d(\theta) \subseteq \F M_d^\dagger(\theta)$ for all $\theta$ satisfying \eqref{5.1}.

Define now $\F M(P; \theta, \eta)$ as the set of those $\Ua \in \T^{r+1}$ having $\ba = \tq^{-1} \B a + \bb \in \F M_d^\dagger(P; \theta)$ and for which there are coprime integers $q_0 \leqslant P^{(2m-1)\eta}$ and $0 \leqslant a_0 < \tq q_0$ satisfying
\begin{align*}
    |\alpha_0 \tq q_0 - a_0| \leqslant P^{-2m + 3(d-1)\theta + (2m-1)\eta},
\end{align*}
where, as customary, the parameter $P$ will usually be suppressed. Then
\begin{align}\label{6.5}
    \vol \F M(\theta, \eta) & \ll \sum_{\tilde q=1}^{P^{2(d-1)\theta}} \left(\prod_{\B j \in J} \sum_{a_{\B j}=0}^{\tilde q -1} \frac{P^{-d+3(d-1)\theta}}{\tilde q} \right) \sum_{q_0=1}^{P^{(2m-1)\eta}} \sum_{a_0=0}^{q_0 \tilde q -1} \frac{P^{-2m + 3(d-1)\theta + (2m-1)\eta}}{q_0 \tilde q}  \nonumber \\
    &\ll P^{-rd -2m + (3r+5)(d-1)\theta + 2(2m-1)\eta}.
\end{align}
Write further $\F m(P; \theta, \eta) = \T^{r+1} \setminus \F M(P; \theta, \eta)$, and observe that, again, one has $|T_P(\Ua)| \ll P^{ms-l \eta+\eps}$ whenever $\Ua \in \F m(P; \theta, \eta)$ with $\ba \in \F M_d^\dagger(P;\theta)$. As in the treatment of \S 3, it is convenient to make the assumption \eqref{3.6}, so we will suppress the parameter $\theta$ in what follows.

\begin{lem}\label{L6.3}
    Suppose that $k$ and $l$ are positive numbers satisfying
    \begin{align}\label{6.6}
        \frac{(3r+5)(d-1)}{k} + \frac{2(2m-1)}{l} < 1.
    \end{align}
  Let further $\theta_*$ be the value of $\theta$ obtained in Lemma~\ref{L6.2}, and suppose that $\theta_*$ satisfies \eqref{5.1} as well as the inequalities
    \begin{align}\label{6.7}
        (k-(3r+2)(d-1))\theta_*>2m
    \end{align}
    and
  \begin{align}\label{6.8}
    (3(d-1) + (k/l)) \theta_* \leqslant 1.
  \end{align}
  Then for any $\eta \in (0,(k/l) \theta_*]$ there exists a $\delta>0$ such that
  \begin{align*}
    \int_{\F m(\eta)} |T_P(\Ua)| \D \Ua \ll P^{ms-rd-2m-\delta}.
  \end{align*}
\end{lem}

\begin{proof}
    The contribution from $\F m(\eta_*)$ can be computed as
    \begin{align*}
         \int_{\F m(\eta_*)} |T_P(\Ua)| \D \Ua        &\ll \int_\T\int_{\F m_d^\dagger(\theta_*)} |T_P(\Ua)| \D \ba \D \alpha_0 + \vol \F M_d^\dagger(\theta_*) \sup_{\Ua \in \F m(\eta_*)} |T_P(\Ua)| \\
        &\ll P^{ms-rd-2m-\delta} + P^{-rd+(3r+2)(d-1)\theta_*}P^{ms-l\eta_*+\eps},
    \end{align*}
    where we used the fact that $\F m_d^\dagger(\theta) \subseteq \F m_d(\theta)$, and the exponent is smaller than $ms-rd-2m$ by \eqref{3.6} and \eqref{6.7}.

    Now in order to bound the contribution from $\F m(\eta) \setminus \F m(\eta_*)$ we fix a sequence $(\eta_i)$ with $T_4 = O(1)$ terms satisfying
    \begin{align*}
      (k/l)\theta_* = \eta_* =\eta_0 > \eta_1 > \dots> \eta_{T_4} = \eta
    \end{align*}
    and
    \begin{align}\label{6.9}
      l\big(\eta_{i-1}- \eta_i) < (l-(3r+5)(d-1)(l/k) - 2(2m-1)\big)\eta.
    \end{align}
    This is possible by \eqref{6.6}, and \eqref{6.8} ensures via \eqref{3.6} that $\eta_* \leqslant 1-3(d-1)\theta_*$.
    Then, arguing as before, we arrive at the bound
    \begin{align*}
        \int_{\F m(\eta) \setminus \F m(\eta_*)} |T_P(\Ua)| \D \Ua &\ll \max_{1 \leqslant i \leqslant T_4} \vol \F M(\eta_{i-1})\sup_{\Ua \in \F m(\eta_i)} |T_P(\Ua)|\\
	   & \ll \max_{1 \leqslant i \leqslant T_4}P^{-rd-2m+((3r+5)(d-1)(l/k) + 2(2m-1)) \eta_{i-1}} P^{ms-l\eta_i+\eps},
    \end{align*}
    where we used \eqref{6.5} and Lemma~\ref{L5.2} (A). Again, by \eqref{6.9}, the exponent is smaller than $ms-rd-2m$ whenever $\eps$ is sufficiently small.
\end{proof}

A straightforward computation shows that the conditions \eqref{6.7} and \eqref{6.8} can be simultaneously satisfied only if
\begin{align}\label{6.10}
    \frac{2m}{l} + \frac{(6m+3r+2)(d-1)}{k} < 1.
\end{align}
Similarly, the conditions \eqref{5.1} and \eqref{6.7} are compatible if
\begin{align}\label{6.11}
    k>(3r+2)(d-1) + \frac{2m(d-1)(r+3)}{d},
\end{align}
and these constraints imply \eqref{6.1} and \eqref{6.2}.
Hence it follows from combining Lemmata \ref{L6.2} and \ref{L6.3} that for every $\eta>0$ there is a $\delta>0$ such that
\begin{align*}
    \int_{\F m(\eta)} |T_P(\Ua)| \D \Ua \ll P^{ms-rd-2m-\delta},
\end{align*}
provided the conditions \eqref{6.6}, \eqref{6.10} and \eqref{6.11} are satisfied.

\section{Major arcs analysis in the case $d>2m$}

This is very similar to the treatment in \S 4. Write $\omega = ((2m-1)+ 3(d-1)(l/k)) \eta$, then after setting $q = \lcm(\tq, q_0)$, we denote by $\F N(q, \U a)$ the set of all $\Ua \in \T^{r+1}$ satisfying
\begin{align*}
    |\alpha_{\B j} - a_{\B j}/q| &\leqslant P^{-d+\omega} \qquad (\B j \in J), & |\alpha_{0} - a_{0}/q| &\leqslant P^{-2m+\omega},
\end{align*}
and
\begin{align*}
    \F N(\eta) = \bigcup_{q=1}^{P^{\omega}} \bigcup_{\substack{\U a=0\\(\U a,q)=1} }^{q-1} \F N(q, \U a).
\end{align*}
As in \S 4, this definition implies that $\F M(\eta) \subseteq \F N(\eta)$, and the volume of these extended major arcs is still estimated by \eqref{4.1} with $\omega$ given as above. Recall the definitions \eqref{4.2}, \eqref{4.3}, \eqref{4.5} and \eqref{4.6}, then \eqref{4.4} and \eqref{4.7} continue to hold with adapted parameters and the error is acceptable if $\eta$ has been chosen small enough.

As in \S 4, we show that the singular integral and the singular series can be extended to infinity. This analysis is in fact very similar to that of the case $d<2m$.
\begin{lem}\label{L7.1}
    We have the bound
    \begin{align*}
        |q^{-ms} S_q(\U a)| \ll q^{\eps}\min \left\{ \left(\frac{q}{(q, \B a)}\right)^{-\frac{k}{d-1}}, q^{-\left(\frac{3(d-1)}{k} + \frac{2m-1}{l}\right)^{-1}} \right\}.
    \end{align*}
\end{lem}

\begin{proof}
    We imitate the proof of Lemma~\ref{L4.1}. Since the lemma is trivially true for $q=1$, we may suppose without loss of generality that $q>1$, and by a similar argument the claim is trivially true if $\B a = \bm 0$, allowing us to assume that $(q, \B a)<q$. Let $Q = q^A$ for some large $A$ to be determined later, and fix $\theta$ such that
    \begin{align}\label{7.1}
        \frac{q}{(q, \B a)} = Q^{(d-1)\theta},
    \end{align}
    so that $\B a/q \in \F M_d(Q; \theta)$. Note that by taking $A$ sufficiently large we may ensure that $\theta<d/(2(d-1))$. Under this assumption, the major arcs are disjoint, so $\B a/q$ is best approximated by itself. Furthermore, in the $q$-aspect it lies just on the edge of the major arcs. As in \S 4, the minor arcs bound continues to apply on the closure of the minor arcs, so together with \eqref{7.1} we find
    $$
        |T_Q(q^{-1} \U a)| \ll Q^{ms+\eps}\left(\frac{q}{(q,\B a )}\right)^{-\frac{k}{d-1}},
    $$
    and on substituting this into \eqref{4.10}, we see that
    \begin{align*}
        |q^{-ms}S_q(\U a)| \ll Q^\eps \left(\frac{q}{(q, \B a)}\right)^{-\frac{k}{d-1}} + q/Q.
    \end{align*}
    Recalling that $Q=q^A$, it is clear that for $A$ sufficiently large the first term dominates. This establishes the first bound in the lemma.

  	Fix now $\eta$ via
  	\begin{align}\label{7.2}
  		q = Q^{(3(d-1)(l/k) + (2m-1))\eta},
  	\end{align}
    so that $q^{-1} \U a \in \F M(Q; \eta)$. By choosing $A$ large enough, we may assume that the major arcs are disjoint. Hence $q^{-1} \U a$ is best approximated by itself, and in the $q$-aspect it lies on the boundary of the major arcs. Using the corresponding minor arcs bound
    $$
        |T_Q(q^{-1} \U a)| \ll Q^{ms+\eps}q^{-\left(\frac{3(d-1)}{k} + \frac{2m-1}{l}\right)^{-1}}
    $$
    together with \eqref{7.2} within \eqref{4.10} yields
    \begin{align*}
    		|q^{-ms}S_q(\U a)| \ll Q^{\eps}q^{-\left(\frac{3(d-1)}{k} + \frac{2m-1}{l}\right)^{-1}} + q/Q,
    \end{align*}
    and we see that for $A$ large enough the first term dominates. This establishes the second statement of the lemma.
\end{proof}

We may now extend $\F S_b(P)$ to infinity. In fact, we have
\begin{align*}
    \sum_{q=1}^{\infty} \sum_{\substack{\U a=0 \\ (\U a,q)=1}}^{q-1} q^{-ms} S_q(\U a) & \ll \sum_{q=1}^{\infty} q^{1-(1-\lambda)\left(\frac{2m-1}{l} + \frac{3(d-1)}{k} \right)^{-1}+\eps} \sum_{d | q} (q/d)^{r-\lambda\left(\frac{d-1}{k}\right)^{-1}}
\end{align*}
for each $\lambda \in [0,1]$. This series converges if, for some $\lambda$, one has
\begin{align*}
    \frac{r(d-1)}{k}< \lambda \quad \text{ and } \quad \frac{2(2m-1)}{l} + \frac{6(d-1)}{k} < 1-\lambda,
\end{align*}
and these inequalities can be simultaneously satisfied if and only if
\begin{align}\label{7.3}
    \frac{2(2m-1)}{l} + \frac{(r+6)(d-1)}{k} <1.
\end{align}

For the treatment of the singular integral we remark that the equations \eqref{4.8} and \eqref{4.9} remain valid with adapted parameters, so it remains to establish an analogous version of Lemma~\ref{L4.2}.

\begin{lem}\label{L7.2}
    We have
	\begin{align*}
        |v_1(\Ub)|\ll \min \big\{1, |\bb|^{-\frac{k}{d-1}+\eps}, |\beta_0|^{-\left( \frac{3(d-1)}{k} + \frac{2m-1}{l}\right)^{-1}+\eps} \big\}.
	\end{align*}
\end{lem}

\begin{proof}
    We imitate again our treatment of the case $2m>d$. The bound $|v_1(\Ub)|\ll1$ is trivial, so we may assume that $|\beta_0|>1$, and also that $|\bb| >1$. Choose $Q=|\Ub|^A$ for some large parameter $A$ to be fixed later, and write $\Ug = (Q^{-d} \bb, Q^{-2m} \beta_0)$, then equations \eqref{4.4} and \eqref{4.8} with $\U a = \UU 0$ and $q = 1$ imply that relation \eqref{4.15} holds true.
    Determine $\theta$ such that $|\bb| = Q^{(d-1)\theta}$. Since for $\theta < d / (2(d-1))$ the major arcs $\F M_d(Q; \theta)$ are disjoint, by choosing $A$ sufficiently large we may ensure that this approximation is the only one, so $\bb$ lies just on the edge of the major arcs. As before, the minor arcs estimate extends to the closure, so by Lemma~\ref{L6.1} (A) we  have $Q^{\theta} \ll (Q^{-ms-\eps} |T_Q(\Ug)|)^{-1/k}$.
    On the other hand, our choice of $\theta$ implies
	\begin{align*}
        |\beta_{\B j}|  \ll Q^{(d-1)\theta} \ll (Q^{-ms-\eps} |T_Q(\Ug)|)^{- \frac{d-1}{k}} \qquad (\B j\in J),
	\end{align*}
    which gives $|T_Q(\Ug)| \ll Q^{ms+\eps} |\bb|^{- \frac{k}{d-1}}$. Inserting this into \eqref{4.15} and recalling $Q = |\Ub|^A$ yields
	\begin{align*}
	   	|v_1(\Ub)| \ll Q^\eps|\bb|^{- \frac{k}{d-1}} + |\Ub|^{1-A},
	\end{align*}
	which is acceptable if $A$ is large enough.
		
	On the other hand, if we fix $\eta$ such that, with \eqref{3.6}, we have
	\begin{align*}
        \max \left\{Q^{-3(d-1)\theta}|\bb|, Q^{-(2m-1)\eta - 3(d-1)\theta} |\beta_0| \right\}=1,
	\end{align*}
    then by choosing $A$ sufficiently large, we can force $\eta$ to be small enough that the major arcs $\F M(q, \U a)$ are disjoint, so $\Ug$ lies on the edge of $\F M(Q; \eta)$, and it follows from extending the minor arcs estimate to the boundary that $Q^{\eta} \ll (Q^{-ms-\eps} |T_Q(\Ug)|)^{-1/l}$. As before, we also have the major arcs information
	\begin{align*}
        |\beta_0| \ll Q^{((2m-1) + 3(d-1)(l/k))\eta} =  \left( Q^{-ms-\eps}|T_Q(\Ug)|\right)^{- \left( \frac{2m-1}{l} + \frac{3(d-1)}{k}\right)}.
	\end{align*}
    This produces an upper bound for $|T_Q(\Ug)|$ which can be substituted into \eqref{4.15} and then yields the desired result, provided that $A$ has been chosen large enough.
\end{proof}

The analysis of the major arcs is now swiftly completed. Again, we define $\chi_{\infty}(b, P, R)$ as in \eqref{4.14} and see that
\begin{align*}
    \chi_{\infty}(b,P,R) & \ll \left(1+\int_1^R  \rho^{-\lambda \left( \frac{d-1}{k}\right)^{-1} +r-1+\eps} \D \rho\right) \left(1+ \int_1^R \rho_0^{-(1-\lambda)\left( \frac{3(d-1)}{k} + \frac{2m-1}{l}\right)^{-1}+\eps}\D \rho_0 \right)
\end{align*}
for any $\lambda \in [0,1]$. The limit $\chi_{\infty}(b, P) = \lim_{R \to \infty }\chi_{\infty}(b,P,R)$ exists if $\lambda$ can be chosen to satisfy
\begin{align*}
    \frac{r(d-1)}{k}< \lambda \quad \text{ and } \quad \frac{(2m-1)}{l} + \frac{3(d-1)}{k} < 1-\lambda,
\end{align*}
 which is possible precisely if
\begin{align}\label{7.4}
    \frac{(2m-1)}{l} + \frac{(r+3)(d-1)}{k} <1.
\end{align}
Observe that both \eqref{7.3} and \eqref{7.4} are implied in \eqref{6.6}. Combining all estimates, we may thus conclude that the asymptotic formula given in \eqref{4.17} holds for $d>2m$, provided the conditions  \eqref{2.4}, \eqref{2.10}, \eqref{6.6}, \eqref{6.10} and \eqref{6.11} are all satisfied.  Again, we may formulate a theorem that is more general than what has been stated in the introduction.

\begin{thm}\label{T7.1}
	Suppose that the conditions
	\begin{align*}
		\frac{2^{d-1}(6m+3r+2)(d-1)}{s-\dim \sing F} + \frac{2^{2m}m}{s} &<1, \\
    	\frac{2^{d-1}(3r+5)(d-1)}{s-\dim \sing F} + \frac{2^{2m}(2m-1)}{s} &<1
	\end{align*}
    and
	\begin{align*}
		s - \dim \sing F > 2^{d-1}(d-1)(3r+2+2m(r+3)/d)
    \end{align*}
	are all satisfied. Then for some $\delta>0$ one has
	\begin{align*}
        N_m(P;b) = P^{ms-rd-2m}  \chi_{\infty}(b, P) \prod_{p \text{ prime}} \chi_p(b) + O(P^{ms-rd-2m-\delta}),
    \end{align*}
    where the factors $\chi_p(b)$ and $\chi_{\infty}(b,P)$ are as in Theorem~\ref{T4.1}.
\end{thm}
Again, the only thing that remains to show is the Euler product representation of the singular series, which follows in a straightforward manner from standard references such as \cite[Chapter 5]{dav}. Remark that for smooth forms $F$ the conditions of Theorem~\ref{T7.1} simplify to
\begin{align*}
		s > \max\{
		&2^{d-1}(6m+3r+2)(d-1) + 2^{2m}m,\\
		&2^{d-1}(3r+5)(d-1) + 2^{2m}(2m-1),\\
        &2^{d-1}(d-1)(3r+2+2m(r+3)/d)
		\}
\end{align*}
and a modicum of computation confirms that for $d>2m$ the first term dominates the second one.

\end{document}